\documentclass[12pt]{amsart}
\usepackage{amssymb,latexsym,amsmath,amscd,amsthm,graphicx, color}
\usepackage[all]{xy}
\usepackage{pgf,tikz}
\usepackage{mathrsfs}
\usetikzlibrary{arrows}
\definecolor{uuuuuu}{rgb}{0.26666666666666666,0.26666666666666666,0.26666666666666666}
\definecolor{xdxdff}{rgb}{0.49019607843137253,0.49019607843137253,1.}
\definecolor{ffqqqq}{rgb}{1.,0.,0.}

\newtheorem{lemma}[subsection]{Lemma}
\newtheorem{hypothesis}[subsection]{Hypothesis}
\newtheorem{prop}[subsection]{Proposition}
\newtheorem{theo}[subsection]{Theorem}

\theoremstyle{definition}
\newtheorem{remark}[subsection]{Remark}
\newtheorem{definition}[subsection]{Definition}
\newtheorem{example}[subsection]{Example}

\theoremstyle{remark}

\newtheorem{proposition}[subsection]{Proposition}
\numberwithin{equation}{section}

\newcommand{\set}[1]{\{#1\}}

\newcommand{\gn}{\nu}

\newcommand{\gs}{\sigma}

\newcommand{\gx}{\xi}

\newcommand{\tit}{\textit}

\newcommand{\C}[1]{\mathcal{#1}}
\newcommand{\D}[1]{\mathbb{#1}}

\newcommand{\te}{\text}

\newcommand{\ti}{\times}

\newcommand{\ep}{\epsilon}

\begin{document}

\title{A study on Quantization Dimension in complete metric spaces}

\author{Mrinal K. Roychowdhury}
\address{School of Mathematical and Statistical Sciences, University of Texas Rio Grande Valley, 1201 West University Drive, Edinburg, TX 78539-2999, USA}
\email{mrinal.roychowdhury@utrgv.edu}

\author{S. Verma}
\address{Department of Mathematics, IIT Delhi, New Delhi, India 110016 }
\email{saurabh331146@gmail.com}






\subjclass[2010]{Primary 28A80; Secondary  37A50, 94A15, 60D05.}
\keywords{Quantization dimension, infinitesimal similitude, IFS}

\begin{abstract}
The primary objective of the present paper is to develop the theory of quantization dimension of an invariant measure associated with an iterated function system consisting of finite number of contractive infinitesimal similitudes in a complete metric space. This generalizes the known results on quantization dimension of self-similar measures in the Euclidean space to a complete metric space. In the last part, continuity of quantization dimension is discussed.
\end{abstract}

\maketitle


.

\section{INTRODUCTION}

Let $(X,d)$ be a complete metric space. Given a Borel probability measure $\mu $ on $X$, a number $r \in (0, +\infty)$ and $ n \in \mathbb{N}$, the $n$th quantization error of order $r$ for $\mu $ is defined by $$V_{n,r}(\mu):=\inf \Big\{\int d(x, A)^r d\mu(x): A \subset X, \, \text{Card}(A) \le n\Big\},$$
where $d(x, A)$ represents the distance of the point $x$ from the set $A$. We define the \textit{quantization dimension} of order $r$ of $\mu $ by $$D_r=D_r(\mu):= \lim_{n \to \infty} \frac{r \log n}{- \log \big(V_{n,r}(\mu)\big)},$$
if the limit exists.
If the limit does not exist, then we define the \tit{lower} and \tit{upper quantization dimensions} by taking limit inferior and limit superior of the sequence, respectively. For $s > 0$, the two numbers
\[\liminf_{n \to \infty} nV^{s/r}(\mu), \te{ and } \limsup_{n \to \infty} nV^{s/r}(\mu)\] are, respectively, called the $s$-dimensional \tit{lower} and \tit{upper quantization coefficients} of order $r$ for $\mu$.
Let $\{\mathbb{R}^m; f_1,f_2, \dots, f_N\}$ be an Iterated Function System(IFS) such that each $f_i $ is a contractive similarity transformation with similarity ratio $c_i.$ Then, by a result of Hutchinson \cite{JH}, for a given probability vector $(p_1,p_2, \dots, p_N)$ there is a unique Borel probability measure $\mu $ satisfying the condition $$ \mu = \sum_{i=1}^{N} p_i  \mu \circ f_i^{-1},$$
and a unique nonempty compact set $E$ such that $$E = \cup_{i=1}^N f_i(E).$$ We call the measure $\mu $ an \tit{invariant measure} associated with the IFS, and the set $E$ the \tit{attractor} or the \tit{limit set} of the given IFS.
An IFS is said to satisfy the \tit{open set condition} (OSC) if there exists a bounded nonempty open set $G$ of $\mathbb{R}^m$ such that $ \cup_{i=1}^N f_i(G) \subset G$ and $ f_i(G) \cap f_j(G) = \emptyset$ for $1 \le i \ne j \le N$, and the \tit{strong OSC} if it also satisfies $E \cap G \ne \emptyset$, where $E$ is the attractor of the IFS. Graf and Luschgy \cite{GL1,GL2} proved that the quantization dimension function $D_r$ of the self-similar measure $\mu $ exists, and satisfies the following equation: $$\sum_{i=1}^N (p_i c_i^r)^{\frac{D_r}{r +D_r}}=1$$
provided that given IFS satisfies the OSC.
\par
Lindsay and Mauldin \cite{LM} generalized the above result to the $F$-conformal measure. They also posed an open question whether or not the lower quantization coefficient is positive, that is,\[\liminf_{n \to \infty} nV^{D_r/r}(\mu) >0,\] where $D_r$ denotes the quantization dimension of the $F$-conformal measure $\mu.$  Zhu settled this question positively \cite{Zhu}. The first author also approached the problem by a different technique \cite{R6}. There has been a large number of papers on quantization dimension by several authors, see, for instance, \cite{GL1,GL2, MR, KPot, R1, R2, R3, R4, R5,  Zador}. We cite the book \cite{GL1} for a comprehensive treatment of this subject. The book \cite{Fal} by Falconer is a good reference for topics related to dimension theory of sets.
\par
In the literature, we find few papers on dimension theory in a general complete metric space. In particular, Schief \cite{Schief} made an attempt to study the dimension theory of self-similar sets in a general complete metric space. Note that the OSC and SOSC are equivalent for a conformal IFS \cite{Peres} and a self-similar IFS \cite{Schief1} on the Euclidean domain. In \cite{Schief}, Schief also proved that the aforesaid statement is not true for a self-similar IFS in a complete metric space. This conveys that we must need some careful observation to develop the dimension theory in a complete metric space. Further, Nussbaum, Priyadarshi and Lunel \cite{Nussbaum1} introduced a notion of infinitesimal similitudes which generalizes the notion of similitudes on a complete metric space, and the notion of conformal mappings on the Euclidean space to a complete metric space. To be precise, they established a formula of the Hausdorff dimension of the limit set (attractor) associated with generalized graph-directed construction consisting of a finite set of contractive infinitesimal similitudes. Motivated and Influenced by the previous papers, we are interested to extend the theory of quantization dimension in a complete metric space. Apart from providing the motivation for this paper, the works reported in \cite{Nussbaum1,R2} also offered us an array of basic techniques which we have modified and adapted. To the best of our knowledge, the continuity of quantization dimension is not discussed so far in the literature. In this paper, we also focus on this issue.
\par
The content of the article is as follows.
The first part of the present paper studies quantization dimension of an invariant measure generated by a finite IFS containing contractive infinitesimal similitudes which are defined on a complete metric space.
The second part of the paper provides some results and examples regarding continuity of quantization dimension.
Thus, our findings in this paper could serve as the first step towards the development of quantization dimension in a complete metric space.

\section{Preliminaries}
An infinitesimal similitude is a generalization of conformal map to the complete metric spaces, see, for instance, \cite{Nussbaum1}.
Let $(X,d)$ be a compact, perfect metric space. A mapping $f: X \rightarrow X $ is said to be an \tit{infinitesimal similitude} at $x \in X$ if for any sequences $(x_m)$ and $(y_m)$ with $x_m \ne y_m$ for $m \in \mathbb{N}$ and $x_m \to x , y_m \to x,$ the limit $$ \lim_{m \to \infty}\frac{d(f(x_m),f(y_m))}{d(x_m,y_m)}=: (Df)(x)$$
exists and is independent of the particular sequences $(x_m)$ and $(y_m)$. We shall say that $f$ is an infinitesimal similitude on $X$ if $f$ is an infinitesimal similitude at $x$ for all $x \in X.$
\begin{lemma}[\cite{Nussbaum1}, Lemma $4.1$]\label{conti}
If $f: X \rightarrow Y$ is an infinitesimal similitude, then $x \to (Df)(x)$ is continuous.
\end{lemma}
\begin{lemma}[\cite{Nussbaum1}, Lemma $4.2$] \label{Chain_lemma}
Let $f: X \rightarrow Y$ and $h: Y \rightarrow Z$ be given. If $f$ is an infinitesimal similitude at $x \in X$ and $h$ is an infinitesimal similitude at $f(x) \in Y$, then $hof$ is an infinitesimal similitude at $x \in X$ and $$ (D(hof))(x)= (Dh)(f(x))(Df)(x).$$
\end{lemma}
\begin{remark}
Let $f:X \to X$ be a similitude. That is, for some fixed constant $c>0,$ $d(f(x),f(y)) =c~d(x,y)$ for every $x,y \in X.$ Then, for any $x \in X,$ and sequences $(x_m),~(y_m)$ with $x_m \ne y_m$ in $X$, we have
\[
(Df)(x)=\lim_{m \to \infty} \frac{d(f(x_m),f(y_m))}{d(x_m,y_m)}=\lim_{m \to \infty} \frac{c~d(x_m,y_m)}{d(x_m,y_m)}= c.
\]
Hence, every similitude is an infinitesimal similitude. Furthermore, if $f$ is also a Lipschitz map with Lipschitz constant $\te{Lip}(f)$, then
\[
(Df)(x)=\lim_{m \to \infty} \frac{d(f(x_m),f(y_m))}{d(x_m,y_m)}\le \lim_{m \to \infty} \frac{\te{Lip}(f)~d(x_m,y_m)}{d(x_m,y_m)}= \te{Lip}(f),
\]
for all $x \in X.$
\end{remark}
\begin{example}
Let $f: \mathbb{R} \to \mathbb{R}$ be a mapping defined by $f(x)= |x|$. Observe that $|f(x)-f(y)| \le |x-y|,~ \forall~ x,y \in \mathbb{R}$, and $f$ is not differentiable at $x=0.$ By taking sequences $x_n= \frac{1}{n},~ y_n =-\frac{1}{n},$ we get $ \lim_{n \to \infty} \frac{\big||x_n| -|y_n|\big|}{|x_n -y_n|}=0.$ For sequences $x_n= \frac{1}{n},~ y_n = 0$, we get $ \lim_{n \to \infty} \frac{\big||x_n| -|y_n|\big|}{|x_n -y_n|}=1.$
Hence, $f$ is not an infinitesimal similitude at $x=0.$
\end{example}
\begin{example}
Define $f: \mathbb{R} \to \mathbb{R}$ by \begin{equation*} f(x) =
       \begin{cases}
     x^2 \sin(\frac{1}{x}),~~ \text{if}~~ x \ne 0 \\
          0 ,~~ \text{otherwise}.
      \end{cases}
   \end{equation*}
   Then, $f$ is differentiable on $\mathbb{R}.$ We have
   \begin{equation*} f'(x) =
          \begin{cases}
        2x \sin(\frac{1}{x})- \cos(\frac{1}{x}),~~ \text{if}~~ x \ne 0 \\
             0 ,~~ \text{otherwise}.
         \end{cases}
      \end{equation*}
    Note that $f'$ is not continuous at $x=0.$ Also, $f$ is not an infinitesimal similitude at $x=0.$
\end{example}
The upcoming theorem is of independent interest. However, it shows that notion of infinitesimal similitude is stronger than differentiability in some sense.
\begin{theo}
Let $f:\mathbb{R} \to \mathbb{R}$ be a differentiable function. Then, the $(Df)$ exists at $x_0$ if and only if the modulus of the derivative $|f'|:\mathbb{R} \to \mathbb{R}$ is continuous at $x_0$. In particular, if $(Df)$ exists, then $(Df)=|f'|.$
\end{theo}
\begin{proof}
Suppose $(Df)$ exists at $x_0$. Then, for $x_n \to x_0$, $x_n \ne x$, $$ (Df)(x_0)= \lim_{n \to \infty} \frac{|f(x_n)-f(x_0)|}{|x_n -x_0|} =\bigg|\lim_{n \to \infty} \frac{f(x_n)-f(x_0)}{x_n -x_0} \bigg|=|f'(x_0)|.$$
By Lemma \ref{conti}, $|f'| $ is continuous at $x_0.$ Now, suppose $|f'|:\mathbb{R} \to \mathbb{R}$ is continuous at $x_0$. Let $x_n \ne y_n$ such that $x_n \to x_0$ and $y_n \to x_0$. By the mean value theorem, $$\bigg|\frac{f(x_n)-f(y_n)}{x_n -y_n}\bigg|= |f'(t_n)|,$$ where either $t_n \in (x_n,y_n)$ or $t_n \in (y_n,x_n).$ Since $x_n \to x_0$ and $y_n \to x_0$, we get $t_n \to x_0$. Hence, due to continuity of $|f'|$ at $x_0,$ we have
\[
|f'(x_0)|= \bigg|\frac{f(x_n)-f(y_n)}{x_n -y_n}\bigg|= (Df)(x_0),
\]
that is, $(Df)(x_0)$ exists. Thus, the proof of the theorem is complete.
\end{proof}

Assume that for $ 1 \le i \le N,$ $ f_i: X \rightarrow X$ are contraction mappings with contraction coefficients $c_i.$ Then, there exists a unique, compact, non-empty set $E \subset X $ with $$ E = \cup_{i=1}^{N}f_i(E).$$ For a detailed study on infinitesimal similitudes and dimension of $E$, the reader is referred to \cite{Nussbaum1}.
 Again, by a result of Hutchinson \cite{JH}, for a given probability vector $(p_1,p_2, \dots, p_N)$ there exists a unique Borel probability measure $\mu $ satisfying the condition $$ \mu = \sum_{i=1}^{N} p_i  \mu \circ f_i^{-1}.$$
Assume that the mappings $f_i:X \rightarrow X$ are infinitesimal similitudes on $X$ and the mappings $x \to (Df_i)(x)$ are strictly positive H\"older continuous functions on $X$ for $1 \le i\le N.$ For $ \sigma \ge 0$ and $0<r< \infty,$ define $L_{\sigma,r}: \mathcal{C}(X) \rightarrow \mathcal{C}(X)$ by $$ (L_{\sigma,r}\Phi)(x)= \sum_{i=1}^{N}\big(p_i(Df_i)(x)^r\big)^{\frac{\sigma}{r + \sigma}} \Phi(f_i(x)).$$
From \cite[Theorem $3.6$]{Nussbaum1} and \cite[Lemma $4.8$]{Nussbaum1}, it follows that $L_{\sigma,r}$ has a strictly positive eigenvector $\Phi_{\sigma,r}$ with eigenvalue equal to the spectral radius $\text{spr}(L_{\sigma,r})$ of $L_{\sigma,r}.$ That is, $L_{\sigma,r}\Phi_{\sigma,r} = \text{spr}(L_{\sigma,r})\Phi_{\sigma,r}.$ Equivalently, we have
\begin{equation}\label{Qeqn0}
 \text{spr}(L_{\sigma,r})\Phi_{\sigma,r}(x)= \sum_{i=1}^{N}\big(p_i(Df_i)(x)^r\big)^{\frac{\sigma}{r + \sigma}} \Phi_{\sigma,r}(f_i(x)).
\end{equation}
\par
Let us define $u_{n,r}(\mu):=\inf \Big\{\int d(x, A \cup G^c)^r d\mu(x): A \subset X, \, \text{Card}(A) \le n \Big\},$ where $G$ is an open set satisfying the SOSC. Note that $$u_{n,r}(\mu)^{1/r} \le V_{n,r}(\mu)^{1/r}:= e_{n,r}(\mu).$$
Let us introduce some notation:
\begin{itemize}
\item $I:=\{1,2,\dots, N\}.$
\item $I^*:= \cup_{n \ge 1} I^n$.
\item For $w \in I^n, n \ge 1,$ define $f_w=f_{w_1}\circ f_{w_2} \circ \dots \circ f_{w_n}$ and $p_w=p_{w_1}p_{w_2}\dots p_{w_n}.$
\item  $|w|:=$the length of word $w$.
\item $T_w=\max \{(Df_w)(x): x \in X\}$ and $R_w=\min \{(Df_w)(x): x \in X\}.$
\item For $w \in I^n,$ that is, $w=w_1w_2\dots w_n$, we write $w^*=w_1w_2 \dots w_{n-1}.$
\item Let $w, \xi \in I^*.$ We use $w \prec \xi$ if there exists $\tau \in I^*$ such that $\xi = w \tau.$
\end{itemize}
\begin{definition}
Let $\Lambda $ be a finite subset of $I^* .$ Then, we call $\Lambda$ a finite maximal antichain if every sequence in $I^{\infty}$ is an extension of some word in $\Lambda$, but no word of $\Lambda$ is an extension of another word in $\Lambda.$
\end{definition}
\begin{definition}
We call sets $A_n \subset X$, for which the values $V_{n,r}(\mu),~u_{n,r}(\mu)$ and $e_{n,r}(\mu)$ are attained, $n$-optimal set for $V_{n,r}(\mu),~u_{n,r}(\mu)$ and $e_{n,r}(\mu)$, respectively.
\end{definition}
The next lemma is essentially \cite[Lemma $4.6$]{Nussbaum1}.
\begin{lemma}\label{QDstrdec}
For a fixed $0< r <+\infty,$ the map $\sigma \to \text{spr}(L_{\sigma,r})$ is continuous and strictly decreasing. Furthermore, there is a unique $\sigma_r \ge 0$ such that $\text{spr}(L_{\sigma_r,r})=1.$
\end{lemma}
With help of the above lemma and Equation \ref{Qeqn0}, the number $\sigma_r$ can be calculated by the following equation:
$$ \lim_{n \to \infty} \frac{1}{n}\log \sum_{|w|=n} (p_w T_w^r)^{\frac{\sigma_r}{r+\sigma_r}}= 0.$$
By similar arguments, we also have $$ \lim_{n \to \infty} \frac{1}{n}\log \sum_{|w|=n} (p_w R_w^r)^{\frac{\sigma_r}{r+\sigma_r}}= 0.$$
We would like to show that $D_r(\mu)=\sigma_r.$ One can follow the paper of Lindsay and Mauldin \cite{LM} to prove results similar to their paper. But, here we are interested in a different approach.
\section{Main result}
Using operator theoretic results, we show the following as our main theorem.
\begin{theo}\label{QDmainthm}

Let $\{X;f_1, f_2, \dots,f_N\}$ be an IFS consisting of contractive infinitesimal similitudes, and the mappings $ x \to (Df_i)(x)$ be strictly positive H\"older continuous functions on $X.$ Let $(p_1,p_2, \dots, p_N)$ be a probability vector and $\mu$ denote the unique invariant measure such that $$ \mu = \sum_{i=1}^{N}p_i\mu \circ f_i^{-1}.$$ Then, the quantization dimension of order $0<r<+\infty$ of $\mu$ is given by the unique $\sigma_r$ such that $\text{spr}(L_{\sigma_r,r})=1,$ provided the SOSC is satisfied by the given IFS.

\end{theo}
\begin{remark}
Since a similarity transformation is an infinitesimal similitude, the above theorem generalizes the result of Graf and Luschgy \cite{GL2}. More precisely, let $\{X;f_1,f_2,\dots,f_N\}$ be an IFS satisfying the SOSC  such that the mappings $f_i$ are similitudes with similarity ratios $c_i.$ Using Equation \ref{Qeqn0} and Theorem \ref{QDmainthm}, we have
\begin{equation*}
\begin{aligned}
\Phi_{\sigma_r,r}(x)& =\text{spr}(L_{\sigma_r,r})\Phi_{\sigma_r,r}(x)= \sum_{i=1}^{N}\big(p_i(Df_i)(x)^r\big)^{\frac{\sigma_r}{r + \sigma_r}}\Phi_{\sigma_r,r}(f_i(x)) \\ & = \sum_{i=1}^{N}\big(p_ic_i^r\big)^{\frac{\sigma_r}{r + \sigma_r}}\Phi_{\sigma_r,r}(f_i(x)) \le \sup_{x \in X}\Phi_{\sigma_r,r}(x) \sum_{i=1}^{N}\big(p_ic_i^r\big)^{\frac{\sigma_r}{r + \sigma_r}} .
\end{aligned}
\end{equation*}
This together with \cite[Lemma $4.8$]{Nussbaum1} implies that $  \sum_{i=1}^{N}\big(p_i~c_i^r\big)^{\frac{\sigma_r}{r + \sigma_r}} \ge 1.$
Similarly, we obtain $  \sum_{i=1}^{N}\big(p_i~c_i^r\big)^{\frac{\sigma_r}{r + \sigma_r}} \le 1.$ Hence, the claim is proved.
\end{remark}

To reach to the proof of our main theorem, we first establish a series of lemmas, remarks, and propositions as follows.
\begin{remark}\label{Qrem2}
Since each mapping $f_w$ is a contraction, we have $ 0< T_w < 1$ for every $w \in I^*.$
\end{remark}
\begin{lemma}\label{Qlem2}
Suppose $X$ and $f_w$ are defined as above. Then,
there exists $C_1 >1 $ such that
$T_w \le C_1 R_w $ for any $w \in I^* .$
\end{lemma}
\begin{proof}
Note that the mappings $ x \mapsto (Df_{\xi})(x) $, $\xi \in I,$ are H\"older continuous with exponent $s>0$, that is, $$ |(Df_{\xi})(x) -(Df_{\xi})(y)| \le K d(x,y)^s,~~ \forall ~~ x,y \in X.$$
Let $x,y \in X.$ Then, using the mean-value theorem, there exists $t$ between $(Df_{\xi})(x)$ and $(Df_{\xi})(y)$ such that
\begin{equation}\label{QDeqnC1}
\begin{aligned}
|\ln((Df_{\xi})(x))- \ln((Df_{\xi})(y))|& = \frac{1}{t} | (Df_{\xi})(x) -(Df_{\xi})(y)| \le \frac{K}{t}  d(x,y)^s \le \frac{K}{m_*} d(x,y)^s.
\end{aligned}
\end{equation}
The last inequality follows because $ 0< m_*:=\min_{\xi \in I}\{R_{\xi}\}\le (Df_{\xi})(x) < 1,~ \forall~ x \in X.$ Let $w \in I^*$ such that $|w|=n$ and $x,y \in X.$ Now, we assign $x_i= f_{w_{i+1}}\circ \dots f_{w_n}(x)$ and $y_i= f_{w_{i+1}}\circ \dots f_{w_n}(y).$ Then, with the notation $c_{\max}=\max\{c_1,c_2, \dots,c_N\}<1$, where $c_{\xi}$ denotes the contraction ratio of $f_{\xi},$ we have $$d(x_i,y_i) \le c_{\max}^{n-i} d(x,y) \le c_{\max}^{n-i} \text{diam}(X).$$
 Thanks to the triangle inequality, Lemma \ref{Chain_lemma} and Equation \ref{QDeqnC1},
 \begin{equation*}
 \begin{aligned}
 & \big| \ln\big((Df_w)(x)\big)- \ln\big((Df_w)(y)\big) \big|\\
  & \le \sum_{i=1}^n \big| \ln\big((Df_{w_i})(x_i)\big)- \ln\big((Df_{w_i})(y_i)\big) \big| \le \sum_{i=1}^n \frac{K}{m_*} d(x_i,y_i)^s \\
   &\le \sum_{i=1}^n \frac{K}{m_*} c_{\max}^{s(n-i)} \text{diam}(X)^s  \le \frac{K \text{diam}(X)^s}{m_*(1 -c_{\max}^s)}.
 \end{aligned}
 \end{equation*}
  Therefore, we deduce $$ (Df_w)(x) \le (Df_w)(y) \exp\Big(\frac{K \text{diam}(X)^s}{m_*(1 -c_{\max}^s)}\Big).$$
Now for a suitable constant $C_1 > 1$, $$ (Df_w)(x) \le C_1 (Df_w)(y)~ \forall ~x,y ~\in X.$$ On taking infimum over all $y \in X$, we have
\begin{equation*}
\begin{aligned}
  (Df_w)(x) & \le C_1 \inf_{y \in X}(Df_w)(y)~ \forall ~x \in X \\ & = C_1 R_w ~ \forall ~x \in X.
\end{aligned}
\end{equation*}
Further, we apply supremum on both sides, and deduce that $ T_w \le C_1 R_w.$
\end{proof}
The above lemma shows that the given IFS consisting of contractive infinitesimal similitudes satisfies the so-called bounded distortion property, see, for instance, \cite[Lemma $2.1$]{Pat}.

\begin{lemma}\label{new22}
Let $(Df_{\xi})(x)> 0 $ for every $x \in X,~~\xi \in I.$ Then, there exists $ \tilde{C} \ge  C_1$ such that for $x,y  \in X$, $$ \frac{(Df_\xi)(x)}{\tilde{C}} \le \frac{d(f_\xi(x),f_\xi(y))}{d(x,y)} \le \tilde{C} (Df_\xi)(x)$$ holds for every $\xi \in I.$ Furthermore, $f_{\xi}^{-1}: f_{\xi}(X) \to X$ is well-defined, and is an infinitesimal similitude on $f_{\xi}(X)$ for each $\xi \in I.$
\end{lemma}
\begin{proof}
Let $\xi \in I.$ We first define a function $\mathcal{F}:X \times X \to \mathbb{R}$ by
\begin{equation*}
\mathcal{F}(x,y)=
\begin{cases}
\frac{d(f_{\xi}(x),f_{\xi}(y))}{d(x,y)}, ~~ \text{if}~~ x \ne y \\
(Df_{\xi})(x),~~~~~\text{if} ~~ x=y.
\end{cases}
\end{equation*}
By Lemma \ref{conti}, one could see that the function $\mathcal{F}$ is continuous on $X \times X.$ Since $(Df_{\xi})(x) > 0 ,$ define $$\mathcal{G}(x,y):= \frac{\mathcal{F}(x,y)}{(Df_{\xi})(x)}.$$
Note that $\mathcal{G}(x,x)=1$ and $\mathcal{G}$ is continuous. Using the compactness of $X \times X$, we can choose a sufficiently large number $\tilde{C} \ge C_1$ such that $$ \tilde{C}^{-1} \le \mathcal{G}(x,y) \le  \tilde{C} $$ for every $x,y \in X \times X.$
Without loss of generality, we assume that $$ \tilde{C}^{-1}~(Df_{\xi})(x) \le \frac{d(f_{\xi}(x),f_{\xi}(y))}{d(x,y)} \le  \tilde{C} (Df_{\xi})(x),$$ holds for all $\xi \in I.$
From the above, we deduce that $f_{\xi}$ is one-one, hence each $f_{\xi}:X \to f_{\xi}(X)$ is invertible. Further, let $z \in f_{\xi}(X).$ Then,
\begin{equation*}
\begin{aligned}
(Df_{\xi}^{-1})(z) & = \lim_{m \to \infty} \frac{d(f_{\xi}^{-1}(z_m),f_{\xi}^{-1}(y_m))}{d(z_m,y_m)}  = \lim_{m \to \infty} \frac{d(x_m,t_m)}{d(f_{\xi}(x_m),f_{\xi}(t_m))}
 = \frac{1}{(Df)(f_{\xi}^{-1}(z))},
\end{aligned}
\end{equation*}
where $x_m=f_{\xi}^{-1}(z_m),~ t_m=f_{\xi}^{-1}(y_m)$ such that $z_m\ne y_m$ and $z_m \to z, ~y_m \to z.$
\end{proof}
Nussbaum et al. \cite{Nussbaum1} proved a result analogous to ``mean value theorem" in terms of infinitesimal similitudes. In particular, Lemma $4.3$ of \cite{Nussbaum1} assumed a condition that $\theta$ is Lipschitz which has not been used in the proof. Therefore, the next lemma can be seen as a modification of \cite[Lemma $4.3$]{Nussbaum1}.
\begin{lemma}\label{new28}
Let $(Df_{\xi})(x)> 0 $ for every $x \in X,~~\xi \in I.$ Then, for a given $\ep>1$ and $w \in I^*$, there exists $\delta> 0$ such that for $x,y  \in X$ with $d(x,y) < \delta$ we have $$ \ep^{-1} (Df_w)(x) \le \frac{d(f_w(x),f_w(y))}{d(x,y)} \le \ep (Df_w)(x).$$
\end{lemma}
\begin{proof}
Let $w \in I^*.$
From the proof of the previous lemma,
 a function $\mathcal{G}:X \times X \to \mathbb{R}$ defined by
 \begin{equation*}
 \mathcal{G}(x,y)=
 \begin{cases}
 \frac{d(f_{w}(x),f_{w}(y))}{d(x,y) (Df_w)(x)}, ~~ \text{if}~~ x \ne y \\
 1,~~~~~\text{if} ~~ x=y.
 \end{cases}
 \end{equation*}
 is uniformly continuous on $X \times X$, where $X \times X$ is equipped with the metric $d_2\big((x,z),(y,t)\big):= \sqrt{d(x,y)^2+d(z,t)^2}.$ Since $\mathcal{G}$ is a strictly positive function, and attains its minimum and maximum, $\ln (\mathcal{G}(x,y))$ is also uniformly continuous on $X \ti X.$ Now, let $\ep>1.$ Then, there exists $\delta> 0$ such that
 \[
    | \ln\big(\mathcal{G}(x,z)\big) -  \ln\big(\mathcal{G}(y,t)\big)| \le \ln(\ep)
  \]
  holds for every $ (x,z), (y,t) \in X \times X$ with $d_2\big((x,z),(y,t)\big)< \delta.$ Putting $z=y$ and $t=y$, we have
  $| \ln\big(\mathcal{G}(x,y) \big)- \ln\big( \mathcal{G}(y,y)\big)| \le \ln(\ep)$ for every $ (x,y), (y,y) \in X \times X$ with $d(x,y)< \delta.$ Consequently,
  \[
   \ln(\ep^{-1}) = - \ln(\ep )\le \ln\big(\mathcal{G}(x,y) \big)\le \ln(\ep)
  \]
  holds for every $x,y \in X$ with $d(x,y)< \delta.$ Thus, the proof of the lemma is complete.
\end{proof}
Here and throughout the paper, we assume the following for perpetuating our general setting of a complete metric space:
\begin{hypothesis}\label{new2}
 There exists $C_2 \ge  C_1$ such that for every $x,y  \in X,~~ w \in I^*$, we have $$ \frac{(Df_w)(x)}{C_2} \le \frac{d(f_w(x),f_w(y))}{d(x,y)} \le C_2 (Df_w)(x).$$
 \end{hypothesis}
Here we emphasis on the fact that the above hypothesis is satisfied by a finite IFS consisting of similitudes \cite{JH}, bi-Lipschitz mappings \cite{R2}, or conformal mappings on Riemannian manifolds \cite{Pat}.
\begin{remark}\label{new2}
In view of the above hypothesis, we immediately get $$ C_2^{-1}  R_w d(x,y)\le d(f_w(x),f_w(y)) \le C_2 T_w d(x,y)$$ for every $x,y \in X$ and $w \in I^*.$
\end{remark}
\begin{remark}
Using Lemma \ref{Qlem2} and the above remark, it follows that with $M= C_2^2$ $$M^{-1} T_w d(x,y) \le d(f_w(x),f_w(y)) \le M T_w d(x,y).$$
\end{remark}
\begin{lemma}\label{QDlem11}
For any $w, \tau \in I^*,$ we have $$M^{-1} T_w T_{\tau} \le T_{w \tau} \le M T_w T_{\tau}.$$
\end{lemma}
\begin{proof}
We have
\begin{equation*}
\begin{aligned}
T_{w\tau} & =\sup_{x \in X} (Df_{w\tau})(x)  = \sup_{x \in X} (Df_w)(f_\tau(x)) ~(Df_\tau)(x)\\  & \le \sup_{x \in X} (Df_w)(f_\tau(x)) ~\sup_{x \in X}(Df_\tau)(x) \le \sup_{x \in X} (Df_w)(x) ~\sup_{x \in X}(Df_\tau)(x) \\
&= T_w T_{\tau}  \le M T_w T_{\tau},
\end{aligned}
\end{equation*}
and
\begin{equation*}
\begin{aligned}
T_{w\tau} & \ge \inf_{x \in X}(Df_{w \tau})(x)  = \inf_{x \in X} (Df_w)(f_\tau(x)) ~(Df_\tau)(x) \\
&\ge  \inf_{x \in X}(Df_w)(f_\tau(x))  \inf_{x \in X}(Df_\tau)(x)  \\
&\ge  \inf_{x \in X}(Df_w)(x)  \inf_{x \in X}(Df_\tau)(x) = R_w R_{\tau} \ge C_1^{-1} T_w C_1^{-1}T_{\tau} \\&  = C_1^{-2}T_w T_{\tau }  \ge M^{-1}T_w T_{\tau}.
\end{aligned}
\end{equation*}
The last inequality in the above follows because $M=C_2^2$ and $C_2 \ge C_1.$ Thus, the assertion is proved.
\end{proof}
\begin{lemma}\label{QDlem12}
Let $0< r<+\infty$ be fixed. Then, for $n \in \mathbb{N},$
$$ M^{-\frac{r \sigma_r}{r+\sigma_r}} \le \sum_{|w|=n} (p_wT_w^r)^{\frac{ \sigma_r}{r+\sigma_r}} \le M^{\frac{r \sigma_r}{r+\sigma_r}}.$$
\end{lemma}
\begin{proof}
Proof follows from Lemma \ref{QDlem11} and \cite[Lemma $3.5$]{R2}.
\end{proof}
\begin{lemma}\label{Qlem3}
Let $0 < r <+\infty$ and $\sigma_r$ be as in the previous section. Then,
$$ \sum_{w \in \Lambda} (p_w T_w^r)^{\frac{\sigma_r}{r+\sigma_r}} \le C^{\frac{2r\sigma_r}{r+\sigma_r}},$$ where $\Lambda$ is a finite maximal antichain.
\end{lemma}
\begin{proof}
Define $n=\min\{|w|: w \in \Lambda\}.$ By using the fact that $\Lambda$ consists of no empty word, we get $n \ge 1.$ Now, for each $w \in \Lambda$ there exists $\xi \in I^*$ such that $|\xi|=n$ and $ w= \xi \tau,$ for some $\tau \in I^*.$ Again, using $p_w \le p_{\xi}$, $T_w \le T_{\xi}$ and Lemma \ref{QDlem12}, we have
 $$ \sum_{w \in \Lambda} (p_w T_w^r)^{\frac{\sigma_r}{r+\sigma_r}} \le  \sum_{w \in \Lambda} (p_\xi T_{\xi}^r)^{\frac{\sigma_r}{r+\sigma_r}} \le \sum_{|\xi|=n} (p_\xi T_{\xi}^r)^{\frac{\sigma_r}{r+\sigma_r}} \le  C^{\frac{2r\sigma_r}{r+\sigma_r}},$$
 and hence the proof.
\end{proof}
\begin{remark}
Note that the set $\Lambda_n:=\{w \in I^*: |w|=n\}$ is a finite maximal antichain and $\mu = \sum_{w \in \Lambda_n} p_w \mu \circ f_w^{-1}.$ On the other hand, for any given finite maximal antichain $\Lambda$, it can be shown (see \cite[Lemma $3.8$]{R2}, \cite{Pat}) that $\mu = \sum_{w \in \Lambda} p_w \mu \circ f_w^{-1}.$ Hence,  our results which are true for the set $\Lambda_n$ , is also true for any general finite maximal antichain $\Lambda.$
\end{remark}
\begin{lemma}\label{Qlem4}
Let $\Lambda$ be a finite maximal antichain and $n \in \mathbb{N}$ with $n \ge \text{Card}(\Lambda)$. Then, for each $0<r<+\infty,$ we have $V_{n,r}(\mu) \le C_2^r \inf\Big\{\sum_{w \in \Lambda}p_w T_w^r~ V_{n_w,r}(\mu): n_w \ge 1, ~ \sum_{w \in \Lambda} n_w \le n\Big\}.$
\end{lemma}
\begin{proof}
We assume $n_w \ge 1,$ for each word $w \in \Lambda,$ and $\Sigma_{w \in \Lambda} n_w \le n.$ Let $A_w$ be an $n_w$-optimal set for $V_{n_w,r}(\mu)$, where $w \in \Lambda.$ Using $\text{Card}\big(\cup f_w(A_w)\big) \le n$ and $ \mu = \Sigma_{w \in \Lambda} p_w \mu \circ f_w^{-1}$, we estimate
\begin{equation*}
\begin{aligned}
V_{n,r}(\mu) & \le \int d(x, \cup f_w(A_w))^r d \mu(x)\\ & = \sum_{w \in \Lambda} p_w \int d(x, \cup f_w(A_w))^r d( \mu \circ f_w^{-1})(x)\\ & \le \sum_{w \in \Lambda} p_w \int d(f_w(x),  f_w(A_w))^r d \mu (x)\\ &
\le C_2^r\sum_{w \in \Lambda} p_w T_w^r \int d(x,  A_w)^r d \mu (x)\\ &
= C_2^r \sum_{w \in \Lambda} p_w T_w^r ~V_{m_w,r}(\mu).
\end{aligned}
\end{equation*}
Thus, the proof of the lemma is complete.
\end{proof}
\begin{prop}\label{Qlem5}
Assume that $0<r<+\infty$, and $\gs_r$ is as defined before. Then, \[\limsup_{n\to \infty} n~ e_{n,r}^{\sigma_r} (\mu)< + \infty.\]
\end{prop}
\begin{proof}
Define $c = \min\{(p_w T_w^r)^{\frac{\sigma_r}{r+ \sigma_r}}: w \in I\}.$ By Remark \ref{Qrem2}, it is simple to see the existence of such a constant $c$ such that $0< c < 1.$ Further, with the help of Lemma \ref{Qlem2}, we have $$ c \le (p_w T_w^r)^{\frac{\sigma_r}{r+ \sigma_r}} \le C_1^{\frac{r \sigma_r}{r+ \sigma_r}} (p_w R_w^r)^{\frac{\sigma_r}{r+ \sigma_r}} .$$
Now, for a given (fixed) $k \in \mathbb{N}$, there exists $n \in \mathbb{N}$ such that $$ \frac{k}{n} (C_1C_2)^{\frac{2r\sigma_r}{r+ \sigma_r}} < c^2.$$
Here we let $c'= c^{-1}   \frac{k}{n} (C_1 C_2)^{\frac{2r\sigma_r}{r+ \sigma_r}} .$ Since $0< c< 1,$ we get $0< c' < 1.$ Consider $\Lambda(c')= \{w \in I^*: (p_w T_w^r)^{\frac{\sigma_r}{r+ \sigma_r}} < c' \le (p_{w^*} T_{w^*}^r)^{\frac{\sigma_r}{r+ \sigma_r}}\} .$ It is easy to observe that $\Lambda(c')$ is a finite subset of $I^*$ and satisfies the definition of a finite maximal antichain. In the light of Lemma \ref{Qlem3}, we obtain
\begin{equation*}
\begin{aligned}
C_2^{\frac{2r\sigma_r}{r+ \sigma_r}} & \ge \sum_{w \in \Lambda(c')} (p_w T_w^r)^{\frac{\sigma_r}{r+ \sigma_r}} \\ & \ge \sum_{w \in \Lambda(c')} (p_w R_w^r)^{\frac{\sigma_r}{r+ \sigma_r}} \\ & \ge \sum_{w \in \Lambda(c')} (p_{w^*} R_{w^*}^r)^{\frac{\sigma_r}{r+ \sigma_r}} (p_{w_{|w|}} R_{w_{|w|}}^r)^{\frac{\sigma_r}{r+ \sigma_r}} \\ & \ge \sum_{w \in \Lambda(c')} C_1^{\frac{-r\sigma_r}{r+ \sigma_r}}(p_{w^*} T_{w^*}^r)^{\frac{\sigma_r}{r+ \sigma_r}} (p_{w_{|w|}} R_{w_{|w|}}^r)^{\frac{\sigma_r}{r+ \sigma_r}} \\ & \ge C_1^{\frac{-2r\sigma_r}{r+ \sigma_r}} cc'~ \text{Card}(\Lambda(c')).
\end{aligned}
\end{equation*}
This in turn yields $$ \text{Card}(\Lambda(c')) \le (cc')^{-1} (C_1 C_2)^{\frac{2r\sigma_r}{r+ \sigma_r}} = \frac{n}{k}.$$
 Therefore, Lemma \ref{Qlem4} produces
\begin{equation*}
\begin{aligned}
V_{n,r}(\mu) & \le C_2^r\sum_{w \in \Lambda(c')} p_w T_w^r V_{k,r}(\mu) \\ & =  C_2^r \sum_{w \in \Lambda(c')} (p_w T_w^r)^{\frac{\sigma_r}{r+ \sigma_r}} (p_w T_w^r)^{\frac{r}{r+ \sigma_r}} V_{k,r}(\mu) \\ & = C_2^r \sum_{w \in \Lambda(c')} (p_w T_w^r)^{\frac{\sigma_r}{r+ \sigma_r}} (c')^{\frac{r}{ \sigma_r}} V_{k,r}(\mu) \\ & \le C_2^{\frac{r^2+3r\sigma_r}{r+ \sigma_r}} (c')^{\frac{r}{ \sigma_r}} V_{k,r}(\mu)\\ & =  C_2^{\frac{r^2+3r\sigma_r}{r+ \sigma_r}} \Big(c^{-1}(C_1 C_2)^{\frac{2r\sigma_r}{r+ \sigma_r}} \frac{k}{n}\Big)^{\frac{r}{ \sigma_r}} V_{k,r}(\mu).
\end{aligned}
\end{equation*}
Consequently, $n V_{n,r}^{\frac{\sigma_r}{r}}(\mu) \le c^{-1}C_1^{\frac{2r\sigma_r}{r+ \sigma_r}} C_2^{3\gs_r} k V_{k,r}^{\frac{\sigma_r}{r}}(\mu).$
The above inequality is true for all $n \in \mathbb{N}$ except finitely many, which further gives $$ \limsup_{n \to \infty } n~ e_{n,r}^{\sigma_r} (\mu)\le c^{-1}C_1^{\frac{2r\sigma_r}{r+ \sigma_r}} C_2^{3\gs_r} k ~e_{k,r}^{\sigma_r}(\mu).$$ Since $k$ is fixed, the last inequality yields the required result.
\end{proof}
\begin{lemma}\label{Qlem6}
Let $\Lambda$ be a finite maximal antichain. Then, there exists $n_0$ depending on $\Lambda$ such that for every $n \ge n_0$ there exists a set of natural numbers $\{n_w:n_w(n)\}_{w \in \Lambda }$ satisfying $\sum_{w \in \Lambda}n_w \leq n$ and
$u_{n, r}(\mu) \ge C_2^{-r} \sum_{w \in \Lambda} p_w R_w^r u_{n_w,r}(\mu).$
\end{lemma}
\begin{proof}
Assume that $G$ is the open set which satisfies the strong open set condition. Since $E \cap G  \ne \emptyset,$ there exists $\xi \in I^*$ such that $f_{\xi}(X) \subset G.$ Let $c= d(f_{\xi}(X),G^c)$ and $R_{\min}= \min\{R_w: w \in \Lambda\}.$ Now, we have $$ d(f_w(f_{\xi}(X)), f_w(G^c)) \ge C_2^{-1} R_w ~ d(f_{\xi}(X),G^c)\ge C_2^{-1} R_{\min} c,$$
which further gives $d(x,G^c)\ge d(x,f_w(G^c)) \ge C_2^{-1} R_{\min} c, $ for every $x \in f_w(f_{\xi}(X)).$ Next, we take $A_n$ as an $n$-optimal set for $u_{n,r}(\mu),$ and $c_n= \max\{d(x, A_n \cup G^c): x \in E\}.$ Then, there exists $n_0$ such that $c_n < C_2^{-1} R_{\min} c$ for every $n \ge n_0, $ because $c_n $ tends to zero as $n$ tends to infinity.
Let $n \ge n_0$ and $x \in f_w(f_{\xi}(E)).$ Then,  there exists $x_a \in A_n \cup G^c$ satisfying the following $$d(x, A_n \cup G^c) = d(x,x_a) \le c_n < C_2^{-1} R_{\min} c.$$ This together with $d(x,f_w(G^c)) \ge C_2^{-1} R_{\min} c, $ implies that $ x_a \in f_w(G).$ Hence, defining $A_{n,w}:= A_n \cap f_w(G),$ we have $n_w=\text{Card}(A_{n,w}) \ge 1$ and $\sum_{w \in I^n} n_w \le n.$ To proceed further, we need to prove the following claim: for any $x \in E$ there exists $x' \in A_{n,w} \cup f_w(G^c)$ such that $d\big(f_w(x),A_{n,w} \cup f_w(G^c)\big) = d(f_w(x),x').$ For the sake of contradiction, assume that the claim is not true. That is, $x' \notin A_{n,w} \cup f_w(G^c).$ Then,  $x' \notin A_n \cap f_w(G)$ and $x' \notin f_w(G^c),$ that is, $x_*:= f_w^{-1}(x') \in X$ is such that $x_* \notin G $ and $ x_* \notin G^c,$ a contradiction. Hence the claim is true, and
\begin{equation*}
\begin{aligned}
d\big(f_w(x),A_{n,w} \cup f_w(G^c)\big) & = d\big(f_w(x), f_w(x_*)\big) \\ & \ge C_2^{-1} R_w ~d(x,x_*) \\ & \ge C_2^{-1} R_w ~d\big(x, f_w^{-1}(A_{n,w})\cup G^c\big).
\end{aligned}
\end{equation*}
By routine calculations, we have
\begin{equation*}
\begin{aligned}
u_{n,r}(\mu) & = \int d\big(x,A_n \cup G^c\big)^r d \mu(x) \\ & = \sum_{w \in \Lambda} p_w \int d(f_w(x),A_n \cup G^c)^r d \mu(x) \\ & \ge \sum_{w \in \Lambda} p_w\int d\big(f_w(x), A_n \cup f_w(G^c)\big)^r d \mu(x) \\ & = \sum_{w \in \Lambda} p_w \int d\big(f_w(x), A_{n_w} \cup f_w(G^c)\big)^r d \mu(x) \\ & \ge C_2^{-r} \sum_{w \in \Lambda} p_w R_w^r \int d\big(x, f_w^{-1}(A_{n_w}\big) \cup G^c)^r d \mu(x) \\ & \ge C_2^{-r} \sum_{w \in \Lambda} p_w R_w^r ~u_{n_w,r}(\mu),
\end{aligned}
\end{equation*}
completing the proof.
\end{proof}
\begin{prop}\label{Qlem7}
Let $0<r < + \infty$ and $ 0 < l< \sigma_r.$ Then,
\[\liminf_{n \to \infty} n e_{n,r}^l (\mu) > 0.\]
\end{prop}
\begin{proof}
In view of the uniqueness of $\sigma_r$, we have $$ \sum_{|w|=n} (p_w R_w^r)^{\frac{l}{r+l}} \to \infty ~ \text{as} ~ n\to \infty,$$
for $0< l < \sigma_r.$ Then,  there exists a sufficiently large number $n \in \mathbb{N}$ such that
\begin{equation}\label{Qeqn1}
 \sum_{|w|=n} (p_w R_w^r)^{\frac{l}{r+l}} \ge 1.
\end{equation}
 With the help of Lemma \ref{Qlem6}, there exists $n_0 \in \mathbb{N}$, and for every $n\ge n_0$ the set of numbers $\{n_w:=n_w(n)\}_{|w|=n}$ satisfying the output of the lemma. Let $c= \min\{n^{r/l}u_{n,r}(\mu): n \le n_0\}.$ It is simple to check that $u_{n,r}(\mu) >0$, and further which implies $c > 0.$ Assume that $m \ge n_0$ and $k^{r/l}u_{k,r}(\mu) \ge c, \forall ~k < n.$ Lemma \ref{Qlem6} produces
\begin{equation*}
\begin{aligned}
n^{r/l}u_{n,r}(\mu) & \ge C_2^{-r} n^{r/l} \sum_{|w|=n} p_w R_w^r u_{n_w,r}(\mu) \\ & = C_2^{-r} n^{r/l} \sum_{|w| =n} p_w R_w^r (n_w)^{-r/l}(n_w)^{r/l} u_{n_w,r}(\mu) \\ & \ge c ~C_2^{-r} \sum_{|w|=n} p_w R_w^r \Big(\frac{n_w}{n}\Big)^{-r/l}.
\end{aligned}
\end{equation*}
Further, appealing H\"older's inequality (with exponent less than $1$), we get
$$ n^{r/l} u_{n,r}(\mu) \ge c ~C_2^{-r} \Big(\sum_{|w|=n} (p_w R_w^r)^{\frac{l}{r+l}}\Big)^{1+\frac{r}{l}} \Big(\sum_{|w|=n} \Big(\frac{n_w}{n}\Big)^{(\frac{-r}{l})(\frac{-l}{r})}\Big)^{\frac{-r}{l}}.$$
Using Equation \ref{Qeqn1} and $\sum_{|w|=n} n_w \le n $, we have $n^{r/l} ~u_{n,r}(\mu) \ge c~ C_2^{-r}.$ Applying the process of induction, we obtain $\liminf_{n \to \infty} n ~u_{n,r}^{l/r}(\mu) \ge c^{l/r}~C_2^{-l} >0.$ This completes the proof.
\end{proof}
We are now ready to prove our main theorem.
\subsection{Proof of Theorem \ref{QDmainthm}}
Proposition~\ref{Qlem5} and a property of upper quantization dimension (see \cite[Proposition $11.3$]{GL1}) yield that $\overline{D_r} \le \sigma_r.$ Further, Proposition~\ref{Qlem7} and a property of lower quantization dimension (see \cite[Proposition $11.3$]{GL1}) produce $\underline{D_r}\ge \sigma_r.$ Hence, the proof is concluded.

\par
\begin{remark} Let $0<r<+\infty$. By Proposition~\ref{Qlem7}, it is known that if $0<l<\gs_r$, then $\liminf_{n \to \infty} n e_{n,r}^l (\mu) > 0$. It is still not known that whether $\liminf_{n \to \infty} n e_{n,r}^{\gs_r} (\mu) > 0$. Similar results can be seen for Gibbs-like
measures on cookie-cutter sets in \cite{R5}. In the sequel, we give an answer of it in the affirmative.
\end{remark}
To prove the aforementioned assertion, we need the next lemma.
\begin{lemma}\label{QDlem77}
Let $0<r<+\infty.$ Then, for every finite maximal antichain $\Lambda,$ we have
\[
   \sum_{w \in \Lambda} \big(p_w R_w^r\big)^{\frac{\sigma_r}{r +\sigma_r}} \ge M^{\frac{-6r \sigma_r}{r +\sigma_r}}.
\]
\end{lemma}
\begin{proof}
Using the definition of finite maximal antichain, we have a finite set of natural numbers $\{n_1,n_2,\dots, n_q\}$ satisfying the conditions $n_1< n_2< \dots< n_q$, and
\[
\Lambda = \Lambda_1 \cup \Lambda_2 \cup \dots \cup \Lambda_q,
\]
where $\Lambda_i=\{w \in \Lambda: |w|=n_i\}$ for every $i \in \{1,2,\dots,q\}.$ Note that for any $w, \gx \in I^\ast$, we have
 \[R_wR_\gx\geq C_1^{-2}T_wT_\gx\geq C_1^{-2} M^{-1}T_{w\gx}\geq M^{-3}T_{w\gx}\geq M^{-3}R_{w\gx}.\]
 Now, choose a natural number $m $ with $m \ge n_q.$ Then, Lemmas \ref{Qlem2} and \ref{QDlem12} dictate
\begin{equation*}
\begin{aligned}
\sum_{w \in \Lambda} (p_w R_w^r)^{\frac{\sigma_r}{r +\sigma_r}} & \ge \sum_{i=1}^q \sum_{w \in \Lambda_i} (p_w R_w^r)^{\frac{\sigma_r}{r +\sigma_r}} M^{\frac{-r \sigma_r}{r +\sigma_r}} \sum_{\xi \in \Lambda_{m-n_i}} (p_\xi R_\xi^r)^{\frac{\sigma_r}{r +\sigma_r}} \\ & \ge M^{\frac{-r \sigma_r}{r +\sigma_r}} \sum_{i=1}^q \sum_{w \in \Lambda_i} (p_w R_w^r)^{\frac{\sigma_r}{r +\sigma_r}}  \sum_{\xi \in \Lambda_{m-n_i}, w \prec \xi} (p_\xi R_\xi^r)^{\frac{\sigma_r}{r +\sigma_r}} \\ & = M^{\frac{-r \sigma_r}{r +\sigma_r}} \sum_{i=1}^q \sum_{w \in \Lambda_i} \sum_{\xi \in \Lambda_{m-n_i}, w \prec \xi} (p_w R_w^r)^{\frac{\sigma_r}{r +\sigma_r}}   (p_\xi R_\xi^r)^{\frac{\sigma_r}{r +\sigma_r}} \\ & \ge  M^{\frac{-4r \sigma_r}{r +\sigma_r}} \sum_{i=1}^q \sum_{w \in \Lambda_i} \sum_{\xi \in \Lambda_{m-n_i}} (p_{w \xi} R_{w \xi}^r)^{\frac{\sigma_r}{r +\sigma_r}} \\ & \ge  M^{\frac{-4r \sigma_r}{r +\sigma_r}}\sum_{w \in \Lambda_m} (p_{w } R_{w }^r)^{\frac{\sigma_r}{r +\sigma_r}}  \\ & \ge M^{\frac{-6r \sigma_r}{r +\sigma_r}}.
\end{aligned}
\end{equation*}
Thus, the proof of the lemma is obtained.
\end{proof}
We are now ready to give the following proposition.
\begin{prop}\label{Qlem71}
Let $0<r < + \infty.$ Then,
\[\liminf_{n \to \infty} n V_{n,r}^{\sigma_r / r} (\mu) > 0.\]
\end{prop}
\begin{proof}
The proof of the proposition can be obtained by using Lemma~\ref{QDlem77}, and the technique of the proof in Proposition \ref{Qlem7}.
\end{proof}

\subsection{Continuity of quantization dimension}
Let $\C P(X)$ denote the set of all Borel probability measures on the compact metric space $(X,d)$. Then,
\begin{equation*}
\begin{aligned}
d_H (\mu, \gn) :=\sup_{\te{Lip}(f)\leq 1} \Big\{\Big|\int_X f d\mu -\int_X fd\gn\Big|\Big\}, \ (\mu, \gn) \in \C M \times \C M, \end{aligned}
\end{equation*}
defines a metric on $\C P(X)$, where $\te{Lip}(f)$ denotes the Lipschitz constant of $f$. Then,
$(\C P(X), d_H)$ is a compact metric space (see \cite[Theorem~5.1]{B}).
Again, we know in the weak topology on $\C P(X)$,
\[\mu_n \to \mu  \Longleftrightarrow \int_X f d\mu_n - \int_X f d\mu \to 0 \te{ for all } f \in \C C(X),\]
where $\C C (X) : =\set{ f : X \to \D R : f \te{ is continuous}}$. It is known that the $d_H$-topology and the weak topology, coincide on the space of probabilities with compact support (see \cite{Mat}).  In our case all measures are compactly supported.

Let us now state and prove the following lemma.

\begin{lemma} \label{lemma45}
Let $\mu_n$ be the Borel probablity measure generated by the finite IFS give by $\{X; f_{n_1}, f_{n, 2}, \cdots, f_{n, N}\}$ associated with the probability vector $(p_{n,1}, p_{n, 2}, \cdots, p_{n,N})$. Let $\mu$ be the Borel probablity measure generated by the finite IFS given by $\{X; f_{1}, f_{2}, \cdots, f_{N}\}$ associated with the probability vector $(p_{1}, p_{2}, \cdots, p_{N})$, i.e.,
\[\mu_n=\sum_{j=1}^N p_{n, j}\mu_n\circ f_{n, j}^{-1}, \te{ and } \mu=\sum_{j=1}^N p_j \mu\circ f_j^{-1}.\]
Suppose that $f_{n, j} \to f_j$, and $p_{n, j} \to p_j$ for all $1\leq j\leq N$.
Then, $\mu_n\to \mu$ as $n\to \infty$.
\end{lemma}
\begin{proof}
We have
\begin{equation*}
\begin{aligned}
& d_H(\mu_n, \mu) =\sup_{\te{Lip}(g)\leq 1} \Big|\int_X g d\mu_n -\int_X g d\mu\Big|\\
&=\sup_{\te{Lip}(g)\leq 1} \Big| \sum_{j=1}^N  p_{n, j}\int_X (g\circ f_{n, j}) d\mu_n- \sum_{j=1}^N  p_j \int_X (g\circ f_j) d\mu\Big|\\
&\leq \sum_{j=1}^N \sup_{\te{Lip}(g)\leq 1} \Big| p_{n, j}\int_X (g\circ f_{n, j}) d\mu_n-p_j \int_X (g\circ f_j) d\mu\Big|\\
&=\sum_{j=1}^N \sup_{\te{Lip}(g)\leq 1} \Big| p_{n, j}\int_X ((g\circ f_{n, j})-(g\circ f_j)) d\mu_n+(p_{n, j}-p_j) \int_X (g\circ f_j) d\mu\Big|\\
&\leq \sum_{j=1}^N \sup_{\te{Lip}(g)\leq 1} \Big| p_{n, j}\int_X ((g\circ f_{n, j})-(g\circ f_j)) d\mu_n\Big|+\sum_{j=1}^N \sup_{\te{Lip}(g)\leq 1}\Big|(p_{n, j}-p_j) \int_X (g\circ f_j) d\mu\Big|\\
&\leq \sum_{j=1}^N \sup_{\te{Lip}(g)\leq 1} p_{n, j}\int_X |(g\circ f_{n, j})-(g\circ f_j)| d\mu_n+\sum_{j=1}^N \sup_{\te{Lip}(g)\leq 1}|(p_{n, j}-p_j| \Big|\int_X (g\circ f_j) d\mu\Big|.
\end{aligned}
\end{equation*}
Since for all $1\leq j\leq N$, $g\circ f_j$ are continuous functions on a compact set, there exists a constant $M_*>0$, such that $|g\circ f_j|\leq M_*$ for all $1\leq j\leq N$.
Since $f_{n, j} \to f_j$, and $p_{n, j} \to p_j$, for any given $\ep>0$, as $\frac \ep{1+NM_*}>0$, there exists a positive integer $\tilde N$, depending on $\ep$ only, such that for all $n\geq \tilde N$, we have
\[
|f_{n, j}-f_j|<\frac{\ep}{1+NM_*}, \te{ and } |p_{n, j}-p_j|<\frac \ep {1+NM_*},
\]
for all $1\leq j\leq N$.
Again, if $\te{Lip}(g)\leq 1$, then for all $n\geq \tilde N$, we have $|(g\circ f_{n, j})-(g\circ f)|\leq |f_{n, j}-f_j|<\frac{\ep}{1+NM_*}$.
Thus, for the above $\ep$, and $\tilde N$, if $n\geq \tilde N $, we have
\begin{equation*}
\begin{aligned}
&d_H(\mu_n, \mu) <\sum_{j=1}^N p_{n, j}\frac{\ep}{1+NM_*}+\sum_{j=1}^N \frac{\ep}{1+NM_*} M_*=\frac{\ep}{1+NM_*}(1+NM_*)=\ep.
\end{aligned}
\end{equation*}
Hence, $\mu_n\to \mu$ as $n\to \infty$. Thus, the proof of the lemma is complete.
\end{proof}
 Recall that $D_r(\mu)$ denotes the quantization dimension of order $r$, for $0<r<+\infty$, of a Borel probability measure $\mu$, and $(Df)$ represents the infinitesimal similitude of a function $f$. Let us now give the following continuity property of quantization dimensions.
\begin{proposition}
Let $\mu_n$, and $\mu$ be the Borel probablity measures given by the IFSs as given in Lemma~\ref{lemma45}, i.e.,
\[
\mu_n=\sum_{j=1}^N p_{n, j}\mu_n\circ f_{n, j}^{-1}, \te{ and } \mu=\sum_{j=1}^N p_j \mu\circ f_j^{-1}.
\]
Suppose that all IFSs are as in Theorem \ref{QDmainthm} such that $f_{n, j} \to f_j$, $p_{n, j} \to p_j$, and $(Df_{n,j})\to (Df_j)$ as $n \to \infty$ for all $1\leq j\leq N$.
Then, $D_r(\mu_n) \to D_r(\mu)$ as $n\to \infty$ provided the SOSC is satisfied by all IFSs.
\end{proposition}
\begin{proof}
 Let $0<r< \infty,~~\sigma \ge 0$ and $n \in \mathbb{N}.$ Define $L_{n,\sigma,r}: \mathcal{C}(X) \rightarrow \mathcal{C}(X)$ by
 \[
 (L_{n,\sigma,r}\Phi)(x)= \sum_{i=1}^{N}\big(p_i~(Df_{n,i})(x)^r\big)^{\frac{\sigma}{r + \sigma}} \Phi(f_{n,r}(x)).
 \]
 By Theorem \ref{QDmainthm}, we have $D_r(\mu_n)= \sigma_{n,r} $ and $D_r(\mu)= \sigma_{r} $, where the numbers $\sigma_{n,r}$ and $\sigma_{r}$ are determined by $spr(L_{n,\sigma_{n,r},r})=1$ and $spr(L_{\sigma_{r},r})=1$ respectively. It suffices to prove that $\sigma_r = \lim_{n \to \infty}\sigma_{n,r}.$ We show it by contradiction. Assume that $\sigma_r \ne \lim_{n \to \infty}\sigma_{n,r}.$ This implies that there exist $\ep_0$ and a subsequence $(n_k)$ such that $\sigma_{n_k,r} > \sigma_r + \ep_0 ~\forall ~k \in \mathbb{N}$ or $\sigma_{n_k,r} < \sigma_r - \ep_0, ~\forall~ k \in \mathbb{N}$. Let us tackle the first condition of the previous line. Then, Lemma \ref{QDstrdec} yields  $$spr(L_{n_k,\sigma_{n_k,r},r})< spr(L_{n_k,\sigma_{r}+\ep_0,r}), ~~\forall~~k \in \mathbb{N}.$$ Recall that under the hypotheses of this theorem, the conclusion of \cite[Lemma $3.2$]{AP} holds, that is,  $\lim_{n \to \infty} spr(L_{n,\sigma,r}) = spr(L_{\sigma,r}).$ Again, by Lemma \ref{QDstrdec} and \cite[Lemma $3.2$]{AP},
 \[
 \lim_{k \to \infty} spr(L_{n_k,\sigma_{r}+\ep_0,r}) = spr(L_{\sigma_{r}+\ep_0,r}) < spr(L_{\sigma_{r},r})=1.
 \]
 Since $spr(L_{n_k,\sigma_{n_k,r},r})=1$ for all $k \in \mathbb{N},$ we have $spr(L_{\sigma_{r}+\ep_0,r}) \ge 1,$ which leads to a contradiction. In the similar lines, we also get a contraction for the other case. Thus, $\sigma_r = \lim_{n \to \infty}\sigma_{n,r}.$
\end{proof}

 In general, $\mu_n \to \mu$ does not imply that $D_r(\mu_n) \to D_r(\mu)$ as $n\to \infty$. In this regard, we give few examples.

\begin{example}
Let $\delta_t$ be the Dirac measure at $t \in \mathbb{R}$, and $\mathcal{L}^1$ denotes one-dimensional Lebesgue measure on $\mathbb{R}.$
Define a sequence of measures $(\mu_n)$ as follows: $\mu_n= \frac{1}{n} \sum_{i=1}^n \delta_{i/n}.$ It can be deduced that $\mu_n \to \mathcal{L}^1|_{[0,1]}.$ However, we show that $D_r(\mu_n)$ does not converge to $D_r\big(\mathcal{L}^1|_{[0,1]}\big).$ We proceed by defining an IFS $\mathcal{I}=\{\mathbb{R}; f_1, f_2\}$, where the mappings $f_i$ are defined as $$ f_1(x)= \frac{x}{2}, ~~~~ \quad ~~ f_2(x)=\frac{x}{2} +\frac{1}{2} \te{ for all } x\in \mathbb R.$$
For the probability vector $(p_1,p_2)=(\frac{1}{2},\frac{1}{2})$, there is a unique measure such that $\mu = p_1 \mu \circ f_1^{-1}+  p_2 \mu \circ f_2^{-1}.$ One observes that $\mu =\mathcal{L}^1|_{[0,1]}.$ Since the IFS $\mathcal{I}$ satisfies the strong open set condition, we deduce that $D_r\big(\mathcal{L}^1|_{[0,1]}\big)=1$ for $0<r<+\infty.$ On the other hand, since each $\mu_n$ has support $\set{1, \frac 12, \cdots, \frac 1 n}$, for any $m\geq n$, we have
\[V_{m, r}(\mu_n)=\inf \Big\{\int d(x, A)^r d\mu_n(x): A \subset X, \, \text{Card}(A) \le m\Big\}=0,\]
by taking $A=\set{1, \frac 12, \cdots, \frac 1 n}$, and hence, $D_r(\mu_n)=0$ for all $n\in \D N$, i.e., $D_r(\mu_n)$ does not converge to $D_r(\mu)$.
\end{example}
\begin{example}
Here we consider a sequence of IFSs $\mathcal{I}_n=\{\mathbb{R};f_{n,1},f_{n,2}\}$, where the mappings involved are defined as $$f_{n,1}(x)=\frac{x}{2n}, ~~\quad~~ f_{n,2}(x)=\frac{x}{2n}+\frac{1}{2n} \te{ for all } x\in \mathbb R. $$
It follows that $(Df_{n,i})(x)= \frac{1}{2n}, ~~\forall~~ x \in \mathbb{R},~~ i=1,2.$ Further, we assume an IFS $\mathcal{I}=\{\mathbb{R};f_{1},f_{2}\}$ with $f_1(x)=f_2(x)=0,~ \forall~~x \in \mathbb{R}.$ With a sequence of probability vectors $(p_{n,1},p_{n,2})=(\frac{1}{2},\frac{1}{2})$, and a probability vector $(p_1,p_2)=(\frac{1}{2},\frac{1}{2}),$ we have $f_{n, i} \to f_i$, $p_{n, i} \to p_i$, and $(Df_{n,i})\to (Df_i)$ as $n \to \infty$ for $i=1,2$. Note that the sequence of invariant measures $(\mu_n)$ associated with $\mathcal{I}_n$ converges to the invariant measure $\mu$ which is associated with $\mathcal{I}.$ Since $\mathcal{I}_n$ satisfy the SOSC, we deduce that $D_r(\mu_n)= \frac{\ln(2)}{\ln(2)+\ln(n)} \to D_r(\mu)=0.$ In this example, the limit IFS $\mathcal{I}$ does not satisfy the SOSC.
\end{example}



\subsection*{Acknowledgements}
   The second author expresses his gratitude to the University Grants
   Commission (UGC), India, for financial support.
\bibliographystyle{amsplain}

\end{document}